\documentclass[twoside, a4paper, 10pt]{amsart}
\title[ ]{Glasner property for linear group actions and their products}
\usepackage{amsfonts}
\usepackage{amsthm}
\usepackage{verbatim}
\usepackage{amsmath, amssymb}
\usepackage{tikz}
\usetikzlibrary{matrix, arrows}
\usepackage{listings}
\usepackage{color}
\usepackage{listings}
\usepackage[all]{xy}
\usepackage[pdftex,colorlinks,linkcolor=blue,citecolor=blue]{hyperref}
\usepackage{graphicx}
\usepackage{float}
\usepackage[margin=3cm]{geometry}
\usepackage{bigints}
\usepackage{dsfont}
\author{Kamil Bulinski}
\address{School of Mathematics and Statistics, University of Sydney, Australia}
\email{kamil.bulinski@sydney.edu.au}

\author{Alexander Fish}
\address{School of Mathematics and Statistics, University of Sydney, Australia}
\email{alexander.fish@sydney.edu.au}

\begin{document}
\maketitle
\raggedbottom

\newcommand{\cA}{\mathcal{A}}
\newcommand{\cB}{\mathcal{B}}
\newcommand{\cC}{\mathcal{C}}
\newcommand{\cD}{\mathcal{D}}
\newcommand{\cE}{\mathcal{E}}
\newcommand{\cF}{\mathcal{F}}
\newcommand{\cG}{\mathcal{G}}
\newcommand{\cH}{\mathcal{H}}
\newcommand{\cI}{\mathcal{I}}
\newcommand{\cJ}{\mathcal{J}}
\newcommand{\cK}{\mathcal{K}}
\newcommand{\cL}{\mathcal{L}}
\newcommand{\cM}{\mathcal{M}}
\newcommand{\cN}{\mathcal{N}}
\newcommand{\cO}{\mathcal{O}}
\newcommand{\cP}{\mathcal{P}}
\newcommand{\cQ}{\mathcal{Q}}
\newcommand{\cR}{\mathcal{R}}
\newcommand{\cS}{\mathcal{S}}
\newcommand{\cT}{\mathcal{T}}
\newcommand{\cU}{\mathcal{U}}
\newcommand{\cV}{\mathcal{V}}
\newcommand{\cW}{\mathcal{W}}
\newcommand{\cX}{\mathcal{X}}
\newcommand{\cY}{\mathcal{Y}}
\newcommand{\cZ}{\mathcal{Z}}
\newcommand{\bA}{\mathbb{A}}
\newcommand{\bB}{\mathbb{B}}
\newcommand{\bC}{\mathbb{C}}
\newcommand{\bD}{\mathbb{D}}
\newcommand{\bE}{\mathbb{E}}
\newcommand{\bF}{\mathbb{F}}
\newcommand{\bG}{\mathbb{G}}
\newcommand{\bH}{\mathbb{H}}
\newcommand{\bI}{\mathbb{I}}
\newcommand{\bJ}{\mathbb{J}}
\newcommand{\bK}{\mathbb{K}}
\newcommand{\bL}{\mathbb{L}}
\newcommand{\bM}{\mathbb{M}}
\newcommand{\bN}{\mathbb{N}}
\newcommand{\bO}{\mathbb{O}}
\newcommand{\bP}{\mathbb{P}}
\newcommand{\bQ}{\mathbb{Q}}
\newcommand{\bR}{\mathbb{R}}
\newcommand{\bS}{\mathbb{S}}
\newcommand{\bT}{\mathbb{T}}
\newcommand{\bU}{\mathbb{U}}
\newcommand{\bV}{\mathbb{V}}
\newcommand{\bW}{\mathbb{W}}
\newcommand{\bX}{\mathbb{X}}
\newcommand{\bY}{\mathbb{Y}}
\newcommand{\bZ}{\mathbb{Z}}

\newcounter{dummy} \numberwithin{dummy}{section}

\theoremstyle{definition}
\newtheorem{mydef}[dummy]{Definition}
\newtheorem{prop}[dummy]{Proposition}
\newtheorem{corol}[dummy]{Corollary}
\newtheorem{thm}[dummy]{Theorem}
\newtheorem{lemma}[dummy]{Lemma}
\newtheorem{eg}[dummy]{Example}
\newtheorem{notation}[dummy]{Notation}
\newtheorem{remark}[dummy]{Remark}
\newtheorem{claim}[dummy]{Claim}
\newtheorem{Exercise}[dummy]{Exercise}
\newtheorem{question}[dummy]{Question}

\newtheorem{theo}{Theorem}
\newtheorem{mainthm}{Theorem}[theo] 
\renewcommand{\themainthm}{\Alph{mainthm}}

\begin{abstract}

A theorem of Glasner from 1979 shows that if $Y \subset \bT = \bR/\bZ$ is infinite then for each $\epsilon > 0$ there exists an integer $n$ such that $nY$ is $\epsilon$-dense. This has been extended in various works by showing that certain irreducible linear semigroup actions on $\bT^d$ also satisfy such a \textit{Glasner property} where each infinite set (in fact, arbitrarily large finite set) will have an $\epsilon$-dense image under some element from the acting semigroup. We improve these works by proving a quantitative Glasner theorem for irreducible linear group actions with Zariski-connected Zariski-closure. This makes use of recent results on linear random walks on the torus. We also pose a natural question that asks whether the cartesian product of two actions satisfying the Glasner property also satisfy a Glasner property for infinite subsets which contain no two points on a common vertical or horizontal line. We answer this question affirmatively for many such Glasner actions by providing a new Glasner-type theorem for linear actions that are not irreducible, as well as polynomial versions of such results.

\end{abstract}

\section{Introduction}

\subsection{Background}

A theorem of Glasner from 1979 \cite{Glasner79} shows that if $Y \subset \bT = \bR/\bZ$ is infinite then for each $\epsilon > 0$ there exists an integer $n$ such that $nY$ is $\epsilon$-dense. A more quantitative version was obtained by Berend-Peres \cite{Berend-Peres}, which states that there exist contstants $c_1,c_2>0$ such that if $Y \subset \bT/\bR$ satisfies $|Y| > \left(c_1/\epsilon\right)^{c_2/\epsilon}$ then $nY$ is $\epsilon$-dense in $\bT$ for some $n \in \bN$. This was improved significantly in the seminal work of Alon-Peres \cite{Alon-Peres} which provided the optimal lower bound as follows.

\begin{thm}[Alon-Peres \cite{Alon-Peres}] For $\delta>0$ there exists $\epsilon_{\delta}>0$ such that for all $0<\epsilon < \epsilon_{\delta}$ and $Y \subset \bT$ with $|Y| > \epsilon^{-2-\delta}$ then there exists $n$ so that $nY$ is $\epsilon$-dense.
\end{thm}

This phenomenom can be extended to other semigroup actions, thus movitating the following definition.

\begin{mydef} Let $G$ be semigroup acting on a compact metric space $X$ by continuous maps. We say that this action is \textit{Glasner} if for all infinite $Y \subset X$ there exists $g \in G$ such that $gY$ is $\epsilon$-dense. Moreover, we say that it is \textit{$k(\epsilon)$-uniformly Glasner} if for all sufficiently small $\epsilon>0$ and $Y \subset X$ with $|Y| >  k(\epsilon)$ we have that there exists $g \in G$ such that $gY$ is $\epsilon$-dense.

\end{mydef}

For instance Kelly- L\^{e} \cite{Kelly-Le} used the techniques of Alon-Peres \cite{Alon-Peres} to show that the natural action of the multiplicative semi-group $M_{d \times d}(\bZ)$ of $d \times d$ integer matrices on $\bT^d$ is $c_d\epsilon^{-3d^2}$uniformly Glasner. This was later improved by Dong in \cite{Dong1} where he showed, using the same techniques of Alon-Peres together with the deep work of Benoist-Quint \cite{Benoist-Quint}, that the action $\operatorname{SL}_d(\bZ) \curvearrowright \bT^d$ is $c_{\delta, d} \epsilon^{-4d - \delta}$-uniformly Glasner for all $\delta >0$. Later Dong \cite{Dong2} used a different technique but still based on the work of Benoist-Quint \cite{Benoist-Quint} to show that a large class of subgroups of $\operatorname{SL}_d(\bZ)$ have the Glasner property.

\begin{thm}[Dong \cite{Dong2}] \label{thm: Zariski-dense Glasner} Let $d \geq 2$ and let $G \leq \operatorname{SL}_d(\bZ)$ be a subgroup that is Zariski dense in $\operatorname{SL}_d(\bR)$. Then $G \curvearrowright \bT^d$ is Glasner, i.e., if $Y \subset \bT^d$ is infinite and $\epsilon>0$ then there exists $g \in G$ such that $gY$ is $\epsilon$-dense.

\end{thm}

We remark that this result, unlike the aforementioned $G = \operatorname{SL}_d(\bZ)$ case in \cite{Dong1}, do not use the techniques of Alon-Peres and do not establish a uniform Glasner property ($Y$ needs to be infinite). A uniform Glasner property was obtained for the case where $G$ acts irreducibly and is generated by finitely many unipotents in \cite{BFGlasner}.

\begin{thm}\label{thm: Glasner irreducible unipotent} Let $d \geq 2 $ and let $G \leq \operatorname{SL}_d(\bZ)$ be a group generated by finitely many unipotent elements such that the representation $G \curvearrowright \bR^d$ is irreducible. Then there exists a constant $C_G>0$ such that $G \curvearrowright \bT^d$ is $\epsilon^{-C_{G}}$-uniformly Glasner, i.e., if $Y \subset \bT^d$ with $|Y|>\epsilon^{-C_G}$ then there exists $g \in G$ such that $gY$ is $\epsilon$-dense.
\end{thm}

Examples of such groups include the subgroup of $\operatorname{SL}_d(\bZ)$ preserving a diagonal quadratic form with coefficients $\pm 1$, not all of the same sign (see \cite{BFGlasner} or Proposition A.5 in \cite{BFTwisted} for more details).

\subsection{Glasner property for groups with Zariski-connected Zariski-closures}

The first main result of this paper extends Theorem~\ref{thm: Glasner irreducible unipotent} by replacing the requirement that $G$ is generated by finitely many unipotent elements by the weaker assumption that $G$ has Zariski-connected Zariski-closure. It also improves Theorem~\ref{thm: Zariski-dense Glasner} by providing a uniform Glasner property and also not requiring the Zariski-closure to be the full $\operatorname{SL}_d(\bR)$.

\begin{mainthm}\label{thm: quantitative Glasner for Zariski-Connected}  Let $G \leq \operatorname{SL}_d(\bZ)$ be a finitely generated group with Zariski-connected Zariski-closure in $\operatorname{SL}_d(\bR)$ such that $G \curvearrowright \bR^d$ is an irreducible representation. Then there exists $C_{G} >0$ such that $G \curvearrowright \bT^d$ is $\epsilon^{-C_{G}}$-uniformly Glasner; i.e., if $Y \subset \bT^d$ with $|Y|>\epsilon^{-C_{G}}$ then there exists $g \in G$ such that $g Y$ is $\epsilon$-dense.

\end{mainthm}

\subsection{Glasner property for product actions}

Let $G_1 \curvearrowright X_1$ and $G_2 \curvearrowright X_2$ be two actions on compact metric spaces that have the Glasner property. We consider the \textbf{product action} \begin{align*}  G_1 \times G_2 & \curvearrowright X_1 \times X_2 \\ 
(g_1,g_2) \cdot (x_1, x_2) &= (g_1x_1, g_2x_2). \end{align*} Clearly, we see that it is not Glasner since a horizontal line $Y = X_1 \times \{x_2\}$ is infinite and $(g_1,g_2)Y \subset X_1 \times \{g_2 x_2\}$ is another horizontal line, which cannot be $\epsilon$-dense for all $\epsilon>0$. The same obstruction occurs if $Y$ is a finite union of horizontal and vertical lines. It is thus natural to ask whether this is the only obstruction by considering infinite sets $Y \subset X_1 \times X_2$ such that no two points are on a common vertical or horizontal line.

\begin{question}[Glasner for product action]\label{question: Glasner product actions} Suppose that $G_1 \curvearrowright X_1$ and $G_2 \curvearrowright X_2$ are Glasner. Suppose $Y \subset X_1 \times X_2$ is an infinite set such that both of the projections onto $X_1$ and $X_2$ are injective on $Y$ (i.e., if $y, y' \in Y$ are distinct then $\pi_{X_1}y \neq \pi_{X_1}y'$ and $\pi_{X_2}y \neq \pi_{X_2}y'$ where $\pi_{X_i}:X_1 \times X_2 \to X_i$ is the projection). Then is it true that for all $\epsilon>0$ there exists $g \in G_1 \times G_2$ such that $gY$ is $\epsilon$-dense in $X_1 \times X_2$ ?

\end{question}

We are unable to find any counterexample so far. The main goal of this paper is to answer this question in the affirmative for many of the semigroups of endomorphisms on $\bT^d$ presented above. We first present a special case of one of our main results, which verifies this for the situation of the original Glasner theorem.

\begin{prop}\label{prop: Glasner 2d} Suppose $Y \subset \bT^2$ is infinite and both of the projections onto the $\bT$ factors are injective on $Y$. Then for all $\epsilon >0$ there exists $(n,m) \in \bN^2$ such that $(n,m)Y$ is $\epsilon$-dense in $\bT^2$. In fact, if $P_1(x),P_2(x) \in \bZ[x]$ are polynomials such that no non-trivial linear combination is constant then for all $\epsilon>0$ there exists $n \in \bN$ such that $(P_1(n), P_2(n))Y$ is $\epsilon$-dense.

\end{prop}

Our next main result demonstrates this phenomenom for the Glasner actions of unipotently generated groups presented in Theorem~\ref{thm: Glasner irreducible unipotent}.

\begin{mainthm}\label{thm: Glasner for product unipotent} Let $G_1 \leq \operatorname{SL}_{d_1}(\bZ)$  and $G_2 \leq \operatorname{SL}_{d_2}(\bZ)$ be subgroups generated by finitely many unipotent elements such that $G_1 \curvearrowright \bR^{d_1}$ and $G_2 \curvearrowright \bR^{d_2}$ are irreducible representations, where $d_1,d_2 \geq 2$ are integers. Then for all $\epsilon>0$ there exists $k \in \bN$ such that if $Y \subset \bT^{d_1} \times \bT^{d_2}$ with $|Y| \geq k$ satisfies that the projections to $\bT^{d_1}$ and $\bT^{d_2}$ are injective on $Y$, then there exists $g \in G_1 \times G_2$ such that $gY$ is $\epsilon$-dense in $\bT^{d_1} \times \bT^{d_2}$.

\end{mainthm}

In light of Theorem~\ref{thm: quantitative Glasner for Zariski-Connected}, it is interesting to ask if the condition that $G_1,G_2$ are generated by unipotent elements can be replaced with the (weaker) assumption that $G_1,G_2$ have Zariski-connected Zariski-closures.

\subsection{Non-irreducible actions}

In the setting of endomorphisms on $\bT^d$, any product action is another action by endomorphisms. Unfortunately, it is not irreducible hence Theorem~\ref{thm: Glasner irreducible unipotent} and Theorem~\ref{thm: quantitative Glasner for Zariski-Connected} do not apply. It is thus natural to ask how one can extend these theorems to the non-irreducible case by placing suitable restrictions on the set $Y$ (in a way that is analogous to the setting in Question~\ref{question: Glasner product actions}). Our next main result achieves this for unipotently generated subgroups.

\begin{mainthm}\label{thm: Glasner affine subspaces}

 Let $G \leq \operatorname{SL}_d(\bZ)$ be a group generated by finitely many unipotent elements. Let $\widetilde{Y} \subset [0,1)^d$ be infinite such that for all distinct $\widetilde{y}, \widetilde{y}' \in \widetilde{Y}$ we have that $\widetilde{y} - \widetilde{y}'$ is not contained in any $G$-invariant proper affine subspace. Then for all $\epsilon>0$ there exists a constant $k$ such that if $|\widetilde{Y}|>k$ then there exists $g \in G$ such that $gY$ is $\epsilon$-dense in $\bT^d$, where $Y \subset \bT^d$ is the projection of $\widetilde{Y}$ onto $\bT^d$.

\end{mainthm}

As before, it is interesting to ask whether this result holds if one replaces the assumption of $G$ being finitely genreated by unipotents with the weaker assumption that the Zariski-closure of $G$ is Zariski-connected.

\begin{proof}[Proof of Theorem~\ref{thm: Glasner for product unipotent} using Theorem~\ref{thm: Glasner affine subspaces}] Let $G = G_1 \times G_2$ and let $\widetilde{Y} \subset [0,1)^{d_1} \times [0,1)^{d_2}$ be a set of representatives for $Y \subset \bT^{d_1} \times \bT^{d_2}$. Let $a \in (\widetilde{Y} - \widetilde{Y}) \setminus \{0 \}$. Using Theorem~\ref{thm: Glasner affine subspaces} it suffies to show that if $Ga \subset W+a$ for some subspace $W \leq \bR^{d_1} \times \bR^{d_2}$ then $W = \bR^{d_1} \times \bR^{d_2}$. To see this, write $a = (a_1,a_2)$ where $a_i \in \bR^{d_i}$ and notice that $a_1,a_2$ are both non-zero (by assumption). Now for $g_1,g_1' \in G_1$ we have that $$(g_1' a_1 - g_1 a_1, 0) = (g_1',1)a - (g_1,1)a \in W.$$ In particular, since $d_1 \geq 2$ and $G_1$ acts irreducibly on $\bR^{d_1}$, we may find $g'_1 \in G_1$ such that $b_1 = g'_1 a_1 - a_1 \neq 0$ (here we use the assumption that $a_1 \neq 0$). Now we have that $$(g_1b_1,0) = (g_1g_1'a_1 - g_1a_1, 0) \in W \quad \text{ for all } g_1 \in G_1.$$ By irreducibility and $b_1 \neq 0$, this means that for all $v_1 \in \bR^{d_1}$ we have that $(v_1,0) \in W$. Similairly, we may show that $(0,v_2) \in W$ for all $v_2 \in \bR^{d_2}$. Thus $W = \bR^{d_1} \times \bR^{d_2}$. \end{proof}

\subsection{Glasner property along polynomial sequences}

Our technique for proving Theorem~\ref{thm: Glasner affine subspaces} extends the polynomial method used in \cite{BFGlasner}. Throughout this paper, we let $\pi_{\bT^d}:\bR^d \to \bT^d$ denote the quotient map.

\begin{mainthm} \label{thm: main polynomial intro} Fix $\epsilon>0$, a positive integer $d$ and let $A(x) \in M_{d \times d}(\bZ[x])$ be a matrix with integer polynomial entries. Then there exists a constant $k = k(\epsilon, A(x), d)>0$ such that the following is true: Suppose $\widetilde{Y} \subset [0,1)^d$ with $|\widetilde{Y}| \geq k$ satisfies the following condition: \begin{align}\label{condition: Glasner difference set} \text{For all } v \in \bZ^d\setminus\{ 0 \} \text{ and distinct } \widetilde{y}, \widetilde{y}' \in \widetilde{Y} \text{ we have that } v \cdot (A(x) - A(0))(\widetilde{y}-\widetilde{y}') \neq 0 \in \bR[x]. \end{align}

Letting $Y = \pi_{\bT^d}(\widetilde{Y})$, there exists $n \in \bZ$ such that $A(n)Y$ is $\epsilon$-dense in $\bT^d$.

\end{mainthm}

\begin{eg}[Proof of Proposition~\ref{prop: Glasner 2d}] Let $P_1(x), P_2(x) \in \bZ[x]$ be polynomials such that no non-trivial linear combination of them is constant and let  $$A(x) = \begin{bmatrix}
    P_1(x)   & 0 \\
    0  & P_2(x) \\
\end{bmatrix} $$ Now suppose $\widetilde{Y} \subset [0,1)^2$ is such that any two distinct $\widetilde{y}, \widetilde{y}'$ are not on a common vertical or horizontal line. This means that $(a_1, a_2) := \widetilde{y} - \widetilde{y}'$ satisfies that $a_1, a_2 \neq 0$. Now the expression (\ref{condition: Glasner difference set}) in Theorem~\ref{thm: main polynomial intro} is $$a_1v_1P_1(x) + a_2v_2 P_2(x) - a_1v_1P_1(0) - a_2v_2P_2(0)$$ where $v = (v_1,v_2) \in \bZ^2 \setminus \{(0,0)\}$. But $(a_1v_1, a_2v_2) \neq (0,0)$ and thus the linear combination $a_1v_1P_1(x) + a_2v_2 P_2(x)$ is a non-constant polynomial and so this expression is non-zero, thus Theorem~\ref{thm: main polynomial intro} applies.

\end{eg}

We remark that the $d=1$ case recovers the result of Berend-Peres \cite{Berend-Peres} (that was later improved quantitatively by Alon-Peres \cite{Alon-Peres}) on the Glasner property along polynomial sequences. More precisely, it states that if $P(x) \in \bZ[x]$ is non-constant then for all $\epsilon>0$ there is a constant $k = k(\epsilon, P(x))$ such that for subsets $Y \subset \bT$ with $|Y|>k$ we have that $P(n)Y$ is $\epsilon$-dense for some $n \in \bZ$.

\begin{eg}[Diagonal action] Consider now the diagonal action $\bN \curvearrowright \bT^2$ given by $n(x,y) = (ny,ny)$. Clearly this is not Glasner since the diagonal (or any non-dense subgroup of $\bT^2$) is an infinite invariant set and hence never becomes $\epsilon$ dense for small enough $\epsilon>0$. However, we may still apply Theorem~\ref{thm: main polynomial intro} to obtain natural assumptions on the set $Y \subset \bT^2$ so that $Y$ has $\epsilon$-dense images under the diagonal action. First, we let  $$A(x) = \begin{bmatrix}
    x   & 0 \\
    0  & x \\
\end{bmatrix} $$ The condition says that for any two distinct $\widetilde{y}, \widetilde{y}' \in \widetilde{Y}$, by setting $(a_1, a_2) := \widetilde{y} - \widetilde{y}'$ we must have $$(a_1v_1 + a_2v_2) x \neq 0 \quad \text{for all } (v_1,v_2) \in \bZ^2 \setminus\{ (0,0) \}.$$

This is equivalent to the statement that $a_1,a_2 \in \bR$ are linearly independent over $\bZ$.

\subsection{Acknowledgement} The authors were partially supported by by the Australian Research Council grant DP210100162.

\end{eg}

\section{Tools}

We now gather some useful tools that have mostly been used in previous works \cite{Kelly-Le}, \cite{Dong1} and \cite{BFGlasner} that are multidimensional generalizations of the techniques originally introduced by Alon-Peres. We restate them for the convenience of the reader, although one slightly new variation will be needed (see Lemma~\ref{lemma: BMV consequence}) mainly for the purposes of proving Theorem~\ref{thm: quantitative Glasner for Zariski-Connected}.

We start with a bound based on \cite{Alon-Peres} that has been extended by the aforementioned works. The following formulation can be found exactly in \cite{BFGlasner} (\cite{Dong1} only demonstrates and uses the $r \geq 1$ case).

\begin{prop} \label{prop: h_q sum bound} Fix an integer $d>0$ and any real number $r>0$. Then there exists a constant $C = C(d,r)$ such that the following is true: Given any distinct $x_1, \ldots, x_k \in \bT^d$ let $h_q$ denote the number of pairs $(i,j)$ with $1 \leq i,j \leq k$ such that $q$ is the minimal (if such exists) positive integer such that $q(x_i - x_j) = 0$. Then $$\sum_{q=2}^{\infty} h_q q^{-r} \leq C k^{2 - r/(d+1)}.$$

\end{prop}

Throughout this paper, we let $e(t) = e^{2\pi i t}$ and we let $$B(M) = \{ \vec{m} \in \bZ^d ~|~ \vec{m} \neq \vec{0} \text{ and } \|\vec{m}\|_{\infty} \leq M \}$$ denote the $L^{\infty}$ ball of radius $M$ in $\bZ^d$ around $\vec{0}$ with $\vec{0}$ removed.

For $u \in \bT^d$ by $|u|$ we will mean the $\| \cdot \|_{\infty}$ distance from $u$ to the origin in $\bT^d$, which may precisely be defined as the distance from the origin in $\bR^d$ to the closest point in the lattice $(\pi_{\bT^d})^{-1}(u) \subset \bR^d$ (this is the metric that we use for $\bT^d$ when defining $\epsilon$-dense).

\begin{thm}[See Corollary 2 in \cite{Barton-Montgomery-Hugh-Valeer}]\label{thm: BMV bound} Let $0<\epsilon< \frac{1}{2}$, $M=\lceil \frac{d}{\epsilon} \rceil $ and $u_1, \ldots, u_k \in \bT^d$ with $|u_i| > \epsilon$ for all $i=1, \ldots, k$. Then  $$\frac{k}{3} \leq \sum_{\vec{m} \in B(M)} \left| \sum_{i=1}^k e(\vec{m} \cdot u_i) \right|.$$ 

\end{thm}

The following is a more relaxed version of Proposition 2 in \cite{Kelly-Le} which we will need for both Theorem~\ref{thm: quantitative Glasner for Zariski-Connected} and Theorem~\ref{thm: Glasner affine subspaces}. It is purely finitistic, rather than asymptotic, which will allow us to take averages with respect to random walks rather than just Ces\`aro averages.

\begin{lemma}\label{lemma: BMV consequence} For integers $d>0$ there exists a constant $C_1=C_1(d)>0$ such that the following is true. Suppose that $g \in M_{d \times d}(\bZ)$ and $x_1,\ldots, x_k$ satisfy that $\{g x_1, \ldots, g x_k\}$ is not $\epsilon$-dense. Then for $M = \lceil \frac{d}{\epsilon} \rceil$ we have $$k^2 < C_1 \epsilon^{-d} \sum_{\vec{m} \in B(M)} \sum_{i,j=1}^k e(\vec{m} \cdot g(x_i - x_j)).$$

\end{lemma}

\begin{proof} Not being $\epsilon$-dense means that there exists $\alpha \in \bT^d$ such that $|\alpha - g x_i |> \epsilon$ for all $i = 1, \ldots k$. Using Theorem~\ref{thm: BMV bound} with $u_i =  g(\alpha - x_i)$ and applying Cauchy-Schwartz we get $$\frac{k^2}{9} \leq |B(M)| \sum_{\vec{m} \in B(M)} \left|  \sum_{i=1}^k e(\vec{m} \cdot(\alpha - g x_i)) \right|^2.$$ Now expanding this square and using the estimate $|B(M)| = (2M+1)^d - 1 = O(2^dd^d\epsilon^{-d})$ gives the result.\end{proof}

\section{Proof of the Glasner property in the case of Zariski-connected Zariski-closures}

Now let $G \leq \operatorname{SL}_d(\bZ)$ be a subgroup with Zariski-connected Zariski closure such that the action of $G$ on $\bR^d$ is irreducible and let $\mu$ be a probability measure on $G$ with finite mean such that $\mu(\{g \}) >0$ for all $g \in G$. Our main tool is the following powerful result on the equidistribution of random linear walks on $\bT^d$ that extends the deep work of Bourgain-Furman-Lindenstrauss-Mozes \cite{BFLM}.

\begin{thm}[See Theorem 1.2 in \cite{HeSaxce2022}]\label{thm: he-saxce} There exists a $\lambda>0$ and a constant $C>0$ such that for every $x \in \bT^d$ and $0<t<\frac{1}{2}$, if $a \in \bZ^d \setminus \{0\}$ is such that $$|\widehat{\mu^{*n} \ast \delta_x}(a)| \geq t \quad \text{and } n \geq C \log \frac{\|a\|}{t}$$ then there exists a $q \in \bZ_{>0}$ and $x' \in \frac{1}{q}\bZ^d/\bZ^d$ such that $$q < \left( \frac{\|a\|}{t} \right)^C \quad \text{ and } d(x,x') \leq e^{-\lambda n}.$$
 
\end{thm}

Letting $n \to \infty$ and taking contrapositives, we obtain the following simple corollary.

\begin{lemma} \label{lemma: HS consequence} There exists a constant $C>0$ such that for every $x \in \bQ^d/\bZ^d$ of the form $x = \frac{1}{q}v$, where $v \in \bZ^d$ and $\operatorname{gcd}(v,q) = 1$, and every $a \in \bZ^d \setminus \{0\}$ we have that $$\limsup_{n \to \infty} |\widehat{\mu^{*n} \ast \delta_x}(a)| \leq 2\|a \| q^{-1/C}.$$ Furthermore, if $y \in \bT^d$ is irrational then $$\limsup_{n \to \infty} |\widehat{\mu^{*n} \ast \delta_y}(a)| = 0.$$

\end{lemma}

\begin{proof} Let $t = \|a \|q^{-1/C}$. If $t \geq \frac{1}{2}$ then the result is clearly true since $|\widehat{\nu} (a)| \leq 1$ for any probability measure $\nu$ on $\bT^d$. On the other hand, if $0<t < \frac{1}{2}$ then we may apply Theorem~\ref{thm: he-saxce} and proceed by contradition to show the sharper (by a factor of $\frac{1}{2}$) bound $$\limsup_{n \to \infty} |\widehat{\mu^{*n} \ast \delta_x}(a)| \leq t.$$ More precisely, if this bound were to fail then we can find large enough $n \geq  C \log \frac{\|a\|}{t}$ so that $|\widehat{\mu^{*n} \ast \delta_x}(a)| \geq t$ and thus there exists $$q' <  \left( \frac{\|a\|}{t} \right)^C = q$$ (meaning that $q' \neq q$) such that $x' \in \frac{1}{q'}\bZ^d/\bZ^d$ such that $d(x,x') < e^{-\lambda n}$. For sufficiently large $n$, this leads to a contradiction as $q \neq q'$.\end{proof}

Intuitively, this can be interpreted as saying that an irrational orbit equidistributes to the Haar measure while the orbit of a rational point with large enough denominator almost equidistributes to the Haar measure. We remark that the proof for $G=\operatorname{SL}_d(\bZ)$ given by Dong in \cite{Dong1} instead used an explicit calculation (Ramanujan sum) for this convolution in the rational case and used the work of Benoist-Quint \cite{Benoist-Quint} for the irrational case.

\begin{proof}[Proof of Theorem~\ref{thm: quantitative Glasner for Zariski-Connected}] Suppose for contradiction that $x_1,\ldots x_k \in \bT^d$ are distinct points such that $\{ g x_1, \ldots, g x_k \}$ is not $\epsilon$-dense in $\bT^d$ for all $g \in G$. Using Lemma~\ref{lemma: BMV consequence}, for $M=\lceil \frac{d}{\epsilon} \rceil $ we have $$k^2 < C_1 \epsilon^{-d} \sum_{\vec{m} \in B(M)} \sum_{i,j=1}^k e(\vec{m} \cdot g(x_i - x_j)) \quad \text{for all } g \in G.$$

Now let $\mu$ be the probability measure on $G$ as above. Integrating this estimate respect to the $n$-fold convolution $\mu^{*n}$ we obtain \begin{align*} k^2 &<  C_1 \epsilon^{-d} \sum_{\vec{m} \in B(M)} \int_{G} \sum_{i,j=1}^k e(\vec{m} \cdot g(x_i - x_j)) d(\mu^{*n})(g) \\ &= C_1 \epsilon^{-d} \sum_{\vec{m} \in B(M)} \sum_{i,j=1}^k \widehat{\mu^{*n} * \delta_{x_i - x_j}}(\vec{m}). \end{align*}

Now using Lemma~\ref{lemma: HS consequence} and letting $n \to \infty$, we get that $$k^2 <  C_1 \epsilon^{-d} \sum_{\vec{m} \in B(M)} \sum_{q=1}^{\infty} h_q \cdot 2 \|\vec{m} \| q^{-1/C} + C_1\epsilon^{-d} k |B(M)|,$$

where $h_q$ denotes the number of pairs $x_i, x_j$ such that $q$ is the least positive integer for which $q(x_i - x_j) = 0$. We apply Proposition~\ref{prop: h_q sum bound} to obtain that
\begin{align*} k^2 &<  2C_1 \epsilon^{-d} \sum_{\vec{m} \in B(M)} C_2k^{2-\frac{1}{C(d+1)}} \|\vec{m} \| + C_1\epsilon^{-d} k |B(M)| \\ &\leq C \epsilon^{-3d} k^{2-\frac{1}{C(d+1)}} +  C\epsilon^{-2d}k \end{align*} for a large enough constant $C$ that depends on $d$ and $G$. Thus, for large enough $k \geq \epsilon^{-C_G}$ for some constant $C_G>0$ this inequality must fail, contradicting the initial assumption that for some distinct $x_1, \ldots, x_k \in \bT^d$ the set $\{ g x_1, \ldots, g x_k \}$ is not $\epsilon$-dense in $\bT^d$ for all $g \in G$.  \end{proof}

\section{Proof of main polynomial theorem}

\begin{lemma}[GCD bound lemma] Let $T_0:\bZ^d \to \bZ^r$ be a $\bZ$-linear transformation. Then there exists a constant $Q=Q(T_0)>0$ and a surjective $\bZ$-linear map $R:\bZ^d \to W$ (where $W \cong \bZ^{d'}$ is an abelian group) such that $T_0 = TR$ for some injective $\bZ$-linear map $T:W \to \bZ^r$ and such that for all $q \in \bZ_{\geq 0}$ we have that $$\operatorname{gcd}(Tw,q) \leq Q \quad \text{for all } w \in W \text{ with } \operatorname{gcd}(w,q)=1.$$

\end{lemma}

\begin{proof} By the Smith Normal Form we may write $$T_0 = LDR'$$ where $L:\bZ^r \to \bZ^r$, $R':\bZ^d \to \bZ^d$ are automorphisms and $D:\bZ^d \to \bZ^r$ is diagonal. This means that $De_i = D_ie_i'$ where $e_i \in \bZ^d$ and $e_i' \in \bZ^r$ is the $i$-th standard basis vector and $D_i \in \bZ$. We also have the divisibility conditions $D_1|D_2|\cdots | D_d$. Now suppose that $k$ is maximal such that $D_k \neq 0$ (thus $D_i = 0$ for all $i > k$). We let $W = \bZ\text{-span}\{e_1, \ldots, e_k\}$. We let $R=P_WR'$ where $P_W:\bZ^d \to W$ is the orthogonal projection and we let $T = (LD)|_{W}:W \to \bZ^r$ be the restriction of $LD$ to $W$. It follows that $$T_0=TR$$ and that $T$ is injective while $R$ is surjective. Indeed, for $a \in \bZ^d$ we write $R'a = w+u$ where $w \in W$ and $u \in \bZ\text{-span}\{e_{i+1}, \ldots, e_d\}$, thus $$DR'a = Dw = DP_WR'a = DRa.$$ Moreover, we see $T$ is injective since $L$ is an automorphism and $D|_W$ is injective.

Now fix $q \in \bZ_{>0}$ and $w \in W$ such that $\operatorname{gcd}(w,q) = 1$. We see that $\operatorname{gcd}(Dw,q) \leq D_k$. Now since $L$ is an automorphism we have that $$\operatorname{gcd}(RDw,q) = \operatorname{gcd}(Dw,q) \leq Q$$ where we set $Q = D_k$. \end{proof}

\begin{proof}[Proof of Theorem~\ref{thm: main polynomial intro}] Suppose $\widetilde{Y} = \{x_1, \ldots, x_k\}$ where the $x_i$ are distinct and suppose that $A(n)Y$ is not $\epsilon$-dense in $\bT^d$ for all $n \in \bZ$ (where $Y = \pi_{\bT^d}(\widetilde{Y})$). So we can apply Lemma~\ref{lemma: BMV consequence} to all such $g \in \{(A(1), \ldots, A(N)\}$ and average over $n=1, \ldots, N$ to obtain that $$k^2 \leq \frac{C_1}{\epsilon^d}\sum_{ \vec{m} \in B(M)} \sum_{1 \leq i,j \leq k} \frac{1}{N} \sum_{n=1}^N e(\vec{m} \cdot A(n) (x_i - x_j)).  $$ Now for each $\vec{m} \in B(M)$ we have a linear map $T_{\vec{m}}:\bR^d \to \bR[x]$ given by $$T_{\vec{m}}u = \vec{m} \cdot (A(x) - A(0))u.$$ Observe that $T_{\vec{m}}$ maps $\bZ^d$ to $\bZ[x]$ and in fact the image of $T_{\vec{m}}$ is isomorphic (as an abelian group) to $\bZ^r$ for some $r \leq D$ where $D$ is the degree of $A(x)$. Using GCD bound lemma above we may write $T_{\vec{m}} = T'_{\vec{m}}R_{\vec{m}}$ where $T'_{\vec{m}}:\bZ^{d'} \to \bZ^d$ is an injective linear map for some $d'\leq d$ and $R_{\vec{m}}:\bZ^d \to \bZ^{d'}$ is surjective and linear. We may also view these maps as integer matrices and thus as linear maps between Euclidean spaces or between Tori. By assumption, we have that $T_{\vec{m}}(x_i - x_j) \in \bR[x]$ is non-zero for distinct $i,j$. Thus $R_{\vec{m}}$ must be injective on $\widetilde{Y}$ hence $|\widetilde{Y}_{\vec{m}}| = k$ where we define $$\widetilde{Y}_{\vec{m}} = R_{\vec{m}} \widetilde{Y} \subset \bR^{d'}.$$ 

Now observe that since there are only finitely many $\vec{m}$ (we consider $\epsilon$ as fixed and $B(M)$ is a finite set) there must exist a constant $L$ such that $$R_{\vec{m}}([0,1)^d) \subset [0,L)^{d'}$$ for all $\vec{m} \in B(M)$. This means that if we set $Y_{\vec{m}} = \pi_{\bT^{d'}} (\widetilde{Y}_{\vec{m}})$ then we must have $$|Y_{\vec{m}}| \geq \frac{|\widetilde{Y}_m|}{L^{d'}} \geq k L^{-d}.$$

Thus we can rewrite our bound as 

\begin{align*} k^2 &\leq \frac{C_1}{\epsilon^d}\sum_{ \vec{m} \in B(M)} \sum_{\widetilde{y},\widetilde{y}' \in \widetilde{Y}_{\vec{m}}} \frac{1}{N} \sum_{n=1}^N e\left((T'_{\vec{m}} (\widetilde{y}-\widetilde{y}')) (n) + \vec{m} \cdot A(0)(\widetilde{y}-\widetilde{y}') \right) \\
&\leq  \frac{C_1}{\epsilon^d} L^{2d} \sum_{ \vec{m} \in B(M)} \sum_{y, y' \in Y_{\vec{m}}} \frac{1}{N} \sum_{n=1}^N e\left((T'_{\vec{m}} (y-y')) (n) + \vec{m} \cdot A(0)(y-y') \right) \end{align*} where the extra $L^{2d}$ factor comes from the fact that a pair $y,y' \in Y_{\vec{m}}$ arises as the projection of at most $L^d L^d$ pairs $\widetilde{y}, \widetilde{y}' \in \widetilde{Y}_{\vec{m}}$. 

Now we consider two cases.

\textbf{Case 1:} $y-y'$ is not rational, i.e., $y-y' \notin \bQ^{d'}/\bZ^{d'}$. We claim that $T'_{\vec{m}}(y-y')(x) \in \bT[x]$ has an irrational non-constant term (the constant term is zero). This follows from basic Linear Algebra: If $A$ is a matrix with entries in $\bQ$ and with trivial kernel then a solution to $Ax = u$, with $u$ a rational vector, must be rational. Thus if  $T'_{\vec{m}}(y-y')(x) \in (\bQ/\bZ)[x]$ then $y-y' \in \bQ^{d'}/\bZ^{d'}$, a contradiction. It now follows by the polynomial Weyl Equidistribution theorem that $$ \lim_{N \to \infty} \frac{1}{N} \sum_{n=1}^N e\left((T'_{\vec{m}} (y-y')) (n) + \vec{m} \cdot A(0)(y-y') \right) = 0.$$

\textbf{Case 2:} $y-y' \in \bQ^{d'}/\bZ^{d'}$. We thus write $y-y' = \frac{w}{q}$ where $w \in \bZ^d$ and $\operatorname{gcd}(w, q) = 1$. We now use the GCD bound lemma to see that $$T'_{\vec{m}}(y-y')(n) = \frac{1}{q}\sum_{j=1}^r b_j n^j$$ where $\operatorname{gcd}(b_1, \ldots, b_r, q) \leq Q(\vec{m})$ for some constant $Q(\vec{m})$ as in the GCD bound lemma. Thus we may apply Hua's bound (see \cite{Hua} or \cite{HuaBook}) to obtain a constant $C_2 = C_2(D, \delta)$ depending only on $D$ and any constant $0 < \delta < \frac{1}{D}$ such that 

\begin{align*} \left | \lim_{N \to \infty} \frac{1}{N}\sum_{n=1}^N e\left((T'_{\vec{m}} (y-y')) (n) + \vec{m} \cdot A(0)(y-y') \right) \right | &= \left| \frac{1}{q}\sum_{n=1}^q e\left(\frac{1}{q}\sum_{j=1}^r b_j n^j + \vec{m} \cdot A(0)(y-y') \right) \right| \\ &\leq C_{2}  \left(\frac{Q(\vec{m})}{q} \right)^{\frac{1}{D} - \delta}.   \end{align*}

Let $Q = \max_{\vec{m} \in B(M)} Q(\vec{m})$. Also, let $h_{q,\vec{m}}$ denote the number of pairs $y,y' \in Y_{\vec{m}}$ such that $y-y' =  \frac{w}{q}$ where $w \in \bZ^d$ and $\operatorname{gcd}(w, q) = 1$. In other words, $h_{q,\vec{m}}$ is the number of pairs $y,y' \in Y_{\vec{m}}$ such that $q$ is the least positive integer for which $q(y-y')=0$. Letting $N \to \infty$ and combining the two cases above we obtain the bound

\begin{align*} k^2 & \leq \frac{C_1 L^{2d}}{\epsilon^d}\sum_{ \vec{m} \in B(M)} \left( \sum_{q=2}^{\infty} h_{q,\vec{m}} C_{2}  \left(\frac{Q}{q} \right)^{\frac{1}{D} - \delta} +k \right) \end{align*}

Now apply Proposition~\ref{prop: h_q sum bound} to get that $$\sum_{q=2}^{\infty} h_{q,\vec{m}} q^{\delta - \frac{1}{D}} \leq C_3 k^{2 - (\frac{1}{D} - \delta)/(d+1)}$$ for some constant $C_3 = C_3(d,D)$ depending only on $d$ and $D$. Thus we have shown that 

$$ k^2 \leq Q^{\frac{1}{D} - \delta}C_2 (2M)^d \frac{C_1L^{2d}}{\epsilon^d}C_3 k^{2 - (\frac{1}{D} - \delta)/(d+1)} + \frac{C_1 L^{2d}}{\epsilon^d}(2M)^d k .$$

Observe that as $\epsilon$, $A(x)$ and $d$ are fixed, we have that $M$, $Q$ and $L$ are fixed and so for large enough $k$ this inequality must fail. In other words, if $|Y|$ is larger than some function of $\epsilon$, $A(x)$ and $d$ then there must exist $n \in \bZ_{\geq 0}$ such that $A(n)Y$ is $\epsilon$-dense in $\bT^d$. \end{proof}

\section{Applications to unipotent subgroups}

\begin{lemma}\label{lemma: hyperplane fleeing Cayley ball} Let $G$ be a semigroup generated by a finite set $U$ and let $$G_n = \{ u_1 \cdots u_r ~|~ 0 \leq r \leq n \text{ and } u_1, \ldots, u_r \in U\}$$ be the ball of radius $n$ in the Cayley Graph of $G$. Suppose that $G$ acts on $\bR^d$ by linear maps and $a \in \bR^d$ satisfies that $Ga$ is not contained in any proper affine subspace. Then $G_d a$ is not contained in any proper affine subspace.

\end{lemma}

\begin{proof}

Let $H_n$ denote the smallest affine subspace containing $G_n a$. In other words, $H_n = W_n + a$ where $$W_n = \bR\text{-span}\{ ga - a ~|~ g \in G_n\}.$$ Clearly $H_n \subset H_{n+1}$. We claim that if $H_N = H_{N+1}$ then $H_n = H_N$ for all $n \geq N$. First note that if $u \in U$ is a generator then $uW_n \subset W_{n+1}$, since for $g \in G_n$ we have that $$u(ga-a) = uga - ua = (uga - a) - (ua - a) \in W_{n+1} +W_1 \subset W_{n+1}.$$ Consequently, we have that for $w \in W_n$ that $$u(w+a) = uw + ua = uw + (ua - a) + a \in W_{n+1} + W_1 + a = W_{n+1} + a$$ and thus $$uH_n \subset H_{n+1}.$$

Thus if $H_N = H_{N+1}$ then $uH_N \subset H_{N+1} = H_N$ for all generators $u$ and thus $gH_N \subset H_N$ for all $g \in G$. Recalling that, by definition, $H_N$ contains $G_N a$ and thus by $G$-invariance $H_N$ contains $G_n a$ for all $n \geq N$, meaning that $H_N$ contains $H_n$ for all $n \geq N$. Thus $H_N = H_n$ for all $n \geq N$. Consequently, the smallest such $N$ for which $H_N = H_{N+1}$ satisfies $N \leq d$ (by dimension arguments). Thus $H_n = H_d$ for all $n \geq d$ which means that $H_d$ contains $G_n a$ for all $n \geq d$ and thus $G a \subset H_d$. By assumption that $Ga$ is not in any proper affine subspace, this means that $H_d = \bR^d$. \end{proof}

\begin{proof}[Proof of Theorem~\ref{thm: Glasner affine subspaces}]

Let $U = \{u_1, \ldots, u_m\}$ be a finite set of generators for $G$ where each $u_i$ is a unipotent element and use cyclic notation so that $u_i = u_{i+jm}$ for all $i, j \in \bZ$. Note that for each fixed $i$ the matrix $u_i ^n$ has entries that are integer polynomials in $n$ hence $$Q_N(n_1, \ldots, n_N) = \prod_{i=1}^{N} u_i^{n_i} \in M_{d \times d}(\bZ[n_1, \ldots, n_N]) $$ is a matrix with multivariate integer polynomial entries in the variables $n_1, \ldots, n_N$. Now let $N = dm$ and use Lemma~\ref{lemma: hyperplane fleeing Cayley ball} to get that $\{ Q_N(n_1, \ldots n_N)a ~|~ n_1, \ldots n_N \in \bZ\}$ is not contained in any proper affine subspace of $\bR^d$ for all fixed non-zero $a \in \widetilde{Y} - \widetilde{Y}$. In other words, for each fixed $a \in (\widetilde{Y} - \widetilde{Y})\setminus\{0\}$ if we let $P_1, \ldots, P_d \in \bR[n_1, \ldots, n_N]$ be the polynomials such that $$Q(n_1, \ldots, n_N)a = (P_1(n_1, \ldots, n_N), \ldots, P_d(n_1, \ldots, n_N))$$ then $P_1, \ldots, P_d, 1$ are linearly independent over $\bR$. But there exists a large enough $R \in \bZ_{>0}$ (independent of $a$) such that the substitutions $n_i \mapsto n_i^{R^{i-1}}$ induce a map $\bZ[n_1, \ldots, n_N] \to \bZ[n]$ that is injective on the monomials appearing in $Q_N(n_1, \ldots, n_N)$. Thus $P_1, \ldots, P_d, 1$ remain linearly independent over $\bR$ after making this substitution, thus $\{ Q(n, n^R, \ldots, n^{R^{N-1}})a ~|~ n \in \bZ \}$ is also not contained in any proper affine subspace. So the proof is complete by applying Theorem~\ref{thm: main polynomial intro} to the polynomial $A(x) = Q(x, x^R, \ldots, x^{R^{N-1}})$, which is independent of $\widetilde{Y}$ and thus the lower bound $k$ is uniform (once $G$ is fixed). \end{proof}

\end{document}